\documentclass{amsart}       % onecolumn (second format)
\usepackage[english]{babel}
\usepackage{latexsym}
\usepackage{amssymb}
\usepackage{amsmath}
\usepackage{verbatim}

\usepackage{verbatim}

\def\fq{\mathbb{F}_q}
\def\cC{\mathcal C}

\def\cQ{\mathcal Q}
\def\cX{\mathcal X}
\def\cY{\mathcal Y}

\def\K{\mathbb{K}}

\def\div{{\rm div}}
\def\fqpr{\mathbb{F}_{q^{'}}}

\newtheorem{lemma}{Lemma}
\newtheorem{theorem}{Theorem}
\newtheorem{proposition}{Proposition}
\newtheorem{corollary}{Corollary}

\begin{document}

\title{Small complete caps from nodal cubics}
\thanks{This research was supported by the Italian Ministry MIUR, 
Geometrie di Galois e strutture di incidenza,
PRIN 2009--2010, 
by INdAM, and by Tubitak Proj. Nr. 111T234.}
\author{Nurdag\"ul Anbar         \and Daniele Bartoli \and Irene Platoni \and
        Massimo Giulietti %etc.
}

%\authorrunning{N.Anbar, D.Bartoli, M.Giulietti, I.Platoni} % if too long for running head

\thanks{Nurdag\"ul Anbar -
Faculty of Engineering and Natural Sciences - Sabanci University\\ Orhanli-Tuzla - 34956 Istanbul - Turkey.  
              \email{nurdagul@su.sabanciuniv.edu}}           %  \\
           
\thanks{Daniele Bartoli -
           Dipartimento di Matematica e Informatica - University of Perugia\\ Via Vanvitelli 1 - 06123 Perugia - Italy.
              \email{bartoli@dmi.unipg.it}}

\thanks{Irene Platoni -
Dipartimento di Matematica - University of Trento\\ Via Sommarive, 14 - 38123, Povo (TN) -  Italy. \email{irene.platoni@unitn.it}}

\thanks{           Massimo Giulietti -
           Dipartimento di Matematica e Informatica - University of Perugia\\ Via Vanvitelli 1 - 06123 Perugia - Italy.
              \email{giuliet@dmi.unipg.it}
}


\maketitle
\begin{abstract}
Bicovering arcs in Galois affine planes of odd order  are a powerful tool for constructing  complete caps in spaces of higher dimensions. 
In this paper we investigate whether some  arcs contained in nodal cubic curves 
are bicovering. For $m_1$, $m_2$ coprime divisors of $q-1$, 
bicovering arcs in $AG(2,q)$ 
of size  $k\le (q-1)\frac{m_1+m_2}{m_1m_2}$
are obtained, provided that 
$(m_1m_2,6)=1$ and $m_1m_2<\sqrt[4]{q}/3.5$.
Such arcs produce complete caps of size 
$kq^{(N-2)/2}$
in affine spaces of dimension $N\equiv 0 \pmod 4$. For infinitely many $q$'s
these caps are the smallest known  complete caps in $AG(N,q)$, $N \equiv 0 \pmod 4$.
%\subclass{51E20}
\end{abstract}

\section{Introduction}

In a (projective or affine)  space over
the finite field with $q$ elements $\fq$, a {\em cap} is a set of points
no three of which collinear; plane caps are usually called {\em arcs}. A cap  is said
to be {\em complete} if it is not contained in a larger cap.
The problem of determining the spectrum of sizes of complete caps in a given space
has been intensively investigated, also  in connection with Coding Theory.
%Interestingly, such problem 
%is related to
%Coding Theory; in fact,  
In fact, complete caps of size $k$ in the $N$-dimensional projective space $PG(N,q)$ with
$k>N+1$ and linear quasi-perfect $[k,k-N-1,4]$-codes over $\fq$
are equivalent objects; see e.g. \cite{GDT}.
%the exceptions of the complete
%$5$-cap in $PG(3,2)$ giving rise to a binary $[5,1,5]$-code, and
%the complete $11$-cap in $PG(4,3)$ corresponding to the Golay
%$[11,6,5]$-code over ${\mathbb F}_3$; see e.g. \cite{GDT}.
%Such codes turn out to have  good covering properties, provided that the size of the related %cap is small with respect to the dimension $N$ and the order $q$  of the ambient space; see %e.g. \cite{GPIEE}.
For fixed $N$ and $q$, the smaller is a complete cap, the better are the covering properties of the corresponding code; see e.g. \cite{GPIEE}.

The trivial lower bound for the size of a complete cap in a Galois space of dimension $N$ and order $q$ is 
\begin{equation}\label{TrivialLB}
\sqrt 2  q^{(N-1)/2}.
\end{equation}
If $q$ is even and $N$ is odd, such bound is substantially sharp; see \cite{MR1395760}. Otherwise, all known infinite families of complete caps have size far from \eqref{TrivialLB}; see the survey papers \cite{HS1,HirschfeldStorme2001} and the more recent works \cite{ABGP,Anb-Giu,BarGiuPrep,MR2656392,Gjac,Gjcd,GPIEE}.

For $q$ odd, a lifting method for constructing complete caps in higher dimensions was proposed in \cite{Gjac}. It is based on the notion of a bicovering arc in a  Galois affine plane $AG(2,q)$; see Section \ref{bicovar} here. An arc $A$ in $AG(2,q)$ is said to be bicovering if every point $P$ off A is covered by at least two secants of $A$,
in such a way that $P$ is external to the segment cut out by one of the secants
but it is internal when the other secant is considered.
As first noticed in \cite{Gjac}, if there exists a bicovering arc of size $k$ in $AG(2,q)$, then a complete cap of size $kq^{(N-2)/2}$ can be constructed in $AG(N,q)$ for each positive $N\equiv 0 \pmod 4$. 

Cubic curves have recently emerged as a useful tool to construct small bicovering arcs \cite{ABGP,Anb-Giu,FaPaSc}. Let $G$ denote 
the abelian group  of the non-singular $\fq$-rational points of an irreducible plane cubic $\cX$ defined over $\fq$. It was already noted by Zirilli \cite{ZIR} that no three points in a coset of a subgroup $K$ of $G$  can be collinear, provided that the index $m$ of $K$ in $G$ is not divisible by $3$. Since then, arcs in cubics have been thoroughly investigated; see e.g. \cite{GFFA,HV,MR792577,MR1221589,TamasRoma,VOL,MR1075538}. It is easily seen that a necessary condition for the union of some cosets of $K$ to be a bicovering arc is that they form a maximal $3$-independent subset of the factor group $G/K$. The possibility of obtaining small bicovering arcs via this method
%if the index $m$ of $K$ is greater than $2$, then a coset of $K$ is never %bicovering; however,  the union of some cosets of $K$ can be a bicovering arc, %provided that they form a maximal $3$-independent subset of the factor group $G/K$. %This possibility 
was investigated in \cite{ABGP} for  $\cX$ singular with a cusp, and in \cite{Anb-Giu} when $\cX$ is a non-singular (elliptic) curve. As a result, general constructions of bicovering arcs of size less than $2pq^{7/8}$ \cite{ABGP} and $q/3$ \cite{Anb-Giu} were provided.

In this paper we deal with cubics with an $\fq$-rational node. We prove that
if $m$ is a divisor of $q-1$ coprime to $6$ and such that $m\le \sqrt[4]{q}/3.5$, 
then the union of the cosets of $K$ corresponding to a maximal $3$-independent subset in $G/K$ is a bicovering arc; see Theorem \ref{pre14mar} in Section \ref{penultima}.
A similar result for $\cX$ elliptic was obtained in \cite{Anb-Giu},  under the further assumption that $m$ is a prime; see \cite[Theorem 1]{Anb-Giu}. 
%However, 
%but with the further assumption that $m$ is a prime. 
%In particular, in \cite{Anb-Giu} the following result is proved:  {\em if $m$ is a %prime divisor of $q-1$ with $7<m<\sqrt[4]{q}/8$, and the cyclic group of order $m$ %admits a maximal $3$-independent subset of size $s$, then there exists a bicovering %arc in $AG(2,q)$ contained in $\cX$ of size approximately $sq/m$}.  
%
%In this paper we achieve a similar result for the case where
%$\cX$ is singular with a node, see Theorem \ref{pre14mar} in Section %\ref{penultima}. 
However, it should be remarked that allowing $m$ to be a composite integer is a crucial improvement. In fact, for $m$ a generic prime, the smallest known maximal $3$-independent subsets in the cyclic group of order $m$ have size approximately $m/3$ \cite{MR1075538}; on the other hand, when $m=m_1m_2$ with $(m_1,m_2)=1$ and $(m,6)=1$, maximal $3$-independent subsets of size less than or equal to $m_1+m_2$ can be easily constructed; see \cite{MR1221589}. 

The main achievements of the present paper are summarized by the following result.
%Our main result is  the following.
%a substantial improvement here is that here it is not required that $m$ is a prime. %Significantly, when $m$ is a composite integer, it is easier to construct smaller %maximal $3$-independent sets in the cyclic group of order $m$. Our main result is  %the following.

\begin{theorem}\label{duedue}
Let $q=p^h$ with $p>3$, and 
let $m$ be a proper divisor of $q-1$ such that $(m,6)=1$ and 
$m\le \frac{\sqrt[4]q}{3.5}$. Assume that  $m=m_1m_2$ with $(m_1,m_2)=1$. Then  
\begin{itemize}
\item[{\rm{(i)}}]
there exists a bicovering arc in $AG(2,q)$ of size less than or equal to
$$
\frac{(m_1+m_2)(q-1)}{m_1m_2};
$$
\item[{\rm{(ii)}}] for $N\equiv 0 \pmod 4$, $N\ge 4$, there exists a complete cap in $AG(N,q)$ of size less than or equal to
$$
\frac{(m_1+m_2)(q-1)}{m_1m_2}q^{\frac{N-2}{2}}.
$$
\end{itemize}
\end{theorem}

When $p$ is large, the value $(m_1+m_2)/m_1m_2$,
where $m_1$, $m_2$ are coprime divisors of $q-1$
with $(m_1m_2,6)=1$ and $m_1m_2<\sqrt[4]{q}/3.5$,
can be significantly smaller than both $2p/q^{1/8}$ and $1/3$. This certainly happens for infinite values of $q$; see  Section \ref{confronto}. In this cases, Theorem \ref{duedue} provides the smallest bicovering arcs known up to now, and hence
% in fact, the  previously known general constructions provide bicovering arcs in %$AG(2,q)$ of size approximately $q/3$  \cite{Anb-Giu} and  $2pq^{\frac{7}{8}}$ %\cite{ABGP}.
 the smallest known complete caps in $AG(N.q)$ with $N\equiv 0 \pmod 4$.

Our proofs heavily rely on concepts and results from both Algebraic Geometry in positive characteristic and  Function Field Theory.  In particular, a crucial role is played by the family of plane curves investigated in Section \ref{sec22}.

%INOLTRE

%It should be noted that Proposition fills a minor gap in the proof....of the completeness of %certain arcs - the smallest

%  we fill a gap in the proof of Szonyi that there exists a complete arc of size roughly %$C(q-1)/m$ (best case: $q^{3/4}$)

%Our final remark here is that 
%for every  $q=p^h$ with $p>3$ there are only three projectively non-equivalent %singular  (irreducible) plane cubics with an $\fq$-rational inflection point; see %e.g. \cite{MR2823118,JAMES}. 
With regard to the problem of constructing bicovering arcs contained in singular cubics, it should be noted that the present paper together with \cite{ABGP} leave just one case open,
namely that of a cubic with an isolated double point, which is currently under investigation by the same authors.   

\section{Preliminaries}

Let $q$ be an odd prime power, and let $\fq$ denote the finite field with $q$ elements. Throughout the paper,  $\K$ will denote the algebraic closure of $\fq$.

\subsection{Complete caps from bicovering arcs}\label{bicovar}

Throughout this section, $N$ is assumed to be a positive integer
divisible by $4$. Let $q'=q^\frac{N-2}{2}$. Fix a basis of $\fqpr$
as a linear space over $\fq$, and identify points in $AG(N,q)$
with vectors of $\fqpr\times \fqpr \times \fq \times \fq$. 
%Also,
%let $\tau$ be a non-square in $\fq$. 

For an arc $A$ in $AG(2,q)$, let
$$C_A=\{(\alpha,\alpha^2,u,v)\in AG(N,q)\mid \alpha \in
\fqpr\,,\,\,\,(u,v)\in A\}\,.$$
As noticed in \cite{Gjac}, the set $C_A$ is a cap whose completeness in $AG(N,q)$ depends on the bicovering properties of $A$ in $AG(2,q)$, defined as follows.
According to Segre \cite{MR0362023}, given three pairwise
distinct  points $P,P_1,P_2$ on a line $\ell$ in $AG(2,q)$, $P$ is
external or internal to the segment $P_1P_2$ depending on whether
\begin{equation}\label{exto}
(x-x_1)(x-x_2)\quad \text{is a non-zero square  or
a non-square in }\fq,
\end{equation}
 where $x$,
$x_1$ and $x_2$ are the coordinates of $P$, $P_1$ and $P_2$ with
respect to any affine frame of $\ell$. 
\begin{def}\label{bico}
Let $A$ be a complete arc in $AG(2,q)$. A point $P\in AG(2,q)\setminus A$ is said to be bicovered by $A$ if there exist $P_1,P_2,P_3,P_4\in A$ such that
$P$ is both external to the segment $P_1P_2$ and internal to the segment $P_3P_4$. If every $P\in AG(2,q)\setminus A$ is bicovered by $A$, then $A$ is said to be a bicovering arc. 
%If there exists precisely one point $Q\in AG(2,q)\setminus A$ which is not bicovered by $A$, %then $A$ is said to be almost bicovering, and $Q$ is called the center of $A$.
\end{def}

A key tool in this paper is the following result from \cite{Gjac}.
\begin{proposition}\label{mainP}
If $A$ is a bicovering $k$-arc, then $C_A$ is a complete cap in $AG(N,q)$ of size $kq^{(N-2)/2}$. 
%If $A$ is almost bicovering with center $Q=(x_0,y_0)$, then either
%$$C=C_A \cup \{(\alpha,\alpha^2-\tau,x_0,y_0)\mid \alpha \in
%\fqpr\}$$ 
%or
%$$C=C_A \cup \{(\alpha,\alpha^2-\tau^2,x_0,y_0)\mid \alpha \in
%\fqpr\}$$ 
%is a complete cap in $AG(N,q)$ of  size $(k+1)q^{(N-2)/2}$. The former case occurs precisely %when $Q$ is external to every secant of $A$ through $Q$.
\end{proposition}

\subsection{Curves and function fields}

%Let $\K$ denote an algebraically closed field and 

Let $\cC$ be a projective absolutely irreducible algebraic curve, defined over the algebraic closure $\K$ of $\fq$. 
%We recall the definition of a function field over $K$ and of its full constant field.
An \textit{algebraic function field $F$ over $\K$} is an extension $F$ of $\K$ such that $F$ is a finite algebraic extension of $\K(x)$, for some element $x\in F$ transcendental over $\K$. If $F=\K(x)$, then $F$ is called the \textit{rational function field over $\K$}.
%The \textit{full constant field of $F$} (also called the \textit{field of constants of $F$}) is the finite %extension $\tilde{K}$ of $K$, consisting of the elements in $F$ that are algebraic over $K$.
%
%We can equivalently say that $K$ is algebraically closed in $F$ or that $K$ is the full constant field of $F$ %when $K=\tilde{K}$.
For basic definitions on function fields we refer to \cite{STI}.

It is well known that to any  curve $\cC$ defined over $\K$ one can associate a function field $\K(\cC)$ over $\K$, namely the field of the rational functions of $\cC$.
Conversely, to a  function field $F$ over $\K$ one can associate a curve $\cC$, defined over $\K$, such that $\K(\cC)$ is $\K$-isomorphic to $F$. The genus of $F$ as a function field coincides with the genus of $\cC$.

A place $\gamma$ of $\K(\cC)$ can be associated to a single point of $\cC$ called the \textit{center} of $\gamma$, but not vice versa. A  bijection between places of $\K(\cC)$ and points of $\cC$ holds provided that the curve $\cC$ is non-singular. 
%We will identify points of $\cC$ with places of $\K(\cC)$, when it is allowed.

Let $F$ be a function field over $\K$. If $F'$ is a finite extension of  $F$, then a place $\gamma '$ of $F'$ is said to be \textit{lying over} a place $\gamma$ of $F$
, %and we write $\gamma'|\gamma$, 
if $\gamma\subset \gamma '$. This holds precisely when $\gamma = \gamma ' \cap F$. In this paper $e\left(\gamma ' | \gamma \right)$ will denote the \textit{ramification index} of $\gamma '$ over $\gamma$.
%, and $f\left(\gamma'| \gamma \right)$ the \textit{relative degree} of $\gamma'$ over $\gamma$, that is the %degree of the extension of the residue class field of $\gamma'$ over the residue class field of $\gamma$.
%If $F$ is a function field over $\fq$, then a rational place $\gamma$ of $F$ is said to {\em split completely} over $F'$ if $e(\gamma'|\gamma)=f(\gamma'|\gamma)=1$ for each $\gamma'$ lying over $\gamma$.

A finite extension $F'$ of a function field $F$ is said to be {\em unramified} if $e(\gamma'|\gamma)=1$ for every  $\gamma'$
place of $F'$ and every $\gamma$ place of $F$ with  $\gamma'$ lying over $\gamma$.

%If $F_1$ and $F_2$ are finite extensions of a function field $F$ over $\fq$, then the \textit{compositum} $F_1 F_2$ is the subfield of the algebraic closure of $F$ generated by $F_1$ and $F_2$. It is possible that the full constant field of $F_1 F_2$ is a proper extension of $\fq$, even when $\fq$ is the full constant field of both $F_1$ and $F_2$. In order to investigate the full constant field of the compositum of two function fields, the following results can be useful.

Throughout the paper, we will refer to the following result a number of times.

\begin{proposition}[Proposition 3.7.3 in \cite{STI}]\label{teo1}
%Let $\K$ be an algebraically closed field.
Let $F$ be an algebraic function field over $\K$, and let $m>1$ be an integer relatively prime to the characteristic of $\K$. Suppose that $u\in F$ is an element satisfying
$
u \neq \omega^e\mbox{ for all }\omega \in F \mbox{ and } e|m\mbox{, }e>1.
$
Let
\begin{equation}
 F'=F(y)\mbox{ with }y^m=u.
\end{equation}
Then
\begin{itemize}
 % \item [i)] The polynomial $\phi(T)=T^n-u$ is the minimal polymonial of $y$ over $F$ (in particular, it is irreducible over $F$). The extension $F'/F$ is Galois of degree $n$; its Galois group is cyclic, and all automorphisms of $F'/F$ are given by $\sigma(y)=\xi y$, where $\xi\in K$ is an $n$-th root of unity.
\item[(i)] for $\gamma'$ a place of $F'$ lying over a place $\gamma$ of $F$, we have
$
e(\gamma'| \gamma)=\frac{m}{r_\gamma}
%\hspace{1cm}\mbox{and}\hspace{1cm}d(\gamma'| \gamma)=\frac{m}{r_P}-1,
$
where
\begin{equation}\label{eq55}
 r_\gamma:=(m,v_{\gamma}(u))>0
\end{equation}
is the greatest common divisor of $m$ and $v_\gamma(u)$;
\item[(ii)] if  $g$ (resp. $g'$) denotes the genus of $F$ (resp. $F'$) as a function field over $\K$, then
$$
g'=1+m\left( g-1+\frac{1}{2}\displaystyle\sum_{\gamma} \left(1-\frac{r_\gamma}{m}\right) \right),
$$
where $\gamma$ ranges over the places of $F$ and $r_\gamma$ is defined by \eqref{eq55}. 
\end{itemize}
\end{proposition}
An extension such as $F'$ in Proposition \ref{teo1} is said to be a {\em Kummer extension} of $F$.

%Now let $\K$ be the algebraic closure of $\fq$.

 A curve $\cC$ is said to be defined over $\fq$ if the ideal of $\cC$ is generated by polynomials with coefficients in $\fq$. 
In this case, $\fq(\cC)$ denotes the subfield of $\K(\cC)$ consisting of the rational functions defined over $\fq$.
% and it is a rational function field over $\fq$. 
%Also, the places of $\fq(\cC)$ correspond to the orbits of points of $\cC$ under the Frobenius map on $\fq$. 
A place of $\K(\cC)$ is said to be $\fq$-rational if it is fixed by the Frobenius map on $\K(\cC)$. The center of an $\fq$-rational place is an  $\fq$-rational point of $\cC$; conversely, if $P$ is a simple $\fq$-rational point of $\cC$, then the only place centered at $P$ is $\fq$-rational. 
The following result is a corollary to Proposition \ref{teo1}.
\begin{proposition}\label{teo1cor}
Let $\cC$ be an irreducible plane curve of genus $g$ defined over $\fq$. Let $u\in \fq(\cC)$ be a non-square in $\K(\cC)$. 
Then the Kummer extension $\K(\cC)(w)$, with $w^2=u$, is the function field of some irreducible curve defined over $\fq$ of genus
$$
g'=2g-1+\frac{M}{2},
%\displaystyle\sum_{\gamma\in \mathbb{P}_F}\left(1-\frac{r_\gamma}{m}\right) \right),
$$
where $M$ is the number of places of $\K(\cC)$ with odd valuation of $u$.
\end{proposition}
The function field $\K(\cC)(w)$ as in Proposition \ref{teo1cor} is said to be a {\em double cover} of $\K(\cC)$ (and similarly the corresponding irreducible curve defined over $\fq$ is called a double cover of $\cC$).

Finally, we recall the Hasse-Weil bound, which will play a crucial role in our proofs.
\begin{proposition}[Hasse-Weil Bound - Theorem 5.2.3 in \cite{STI}]\label{HaWe}
The number $N_q$ of $\fq$-rational places of the function field $\K(\cC)$ of a curve $\cC$ defined over $\fq$ with genus $g$ satisfies
$$
|N_q-(q+1)|\le 2g \sqrt q.
$$
\end{proposition}

% if the corresponding point on $\cC$ is rational. Moreover, a rational place $\gamma$ of $\fq(\cC)$ is said to %{\em split completely} over a finite extension of $\fq(\cC)$ if $e(\gamma'|\gamma)=f(\gamma'|\gamma)=1$ for %each $\gamma'$ lying over $\gamma$.

\subsection{Order and class of a place with respect to a plane model}\label{order}

 Let $\mathcal{C}$ be the algebraic plane curve defined by the equation
$f(X,Y)=0$, where $f(X,Y)$ is an irreducible polynomial over $\K$, and let
$\K(\mathcal{C})$ be the function field
of $\mathcal{C}$. Let $\bar x$ and $\bar y$ denote the rational functions associated to the affine coordinates $X$ and $Y$, respectively. Then $\K(\cC)=\K(\bar{x},\bar{y})$ with $f(\bar x,\bar y)=0$. Let $\mathbb P_{\mathcal C}$ denote the set of all places of $\K(\mathcal C)$, and let
$ \mathrm{Div}(\K(\mathcal{C}))$ be the group of divisors of $\K(\mathcal C)$, that is the free abelian group generated by $\mathbb P_{\mathcal C}$.

Let $\mathcal{D}$
be the following subset of $ \mathrm{Div}(\K(\mathcal{C}))$:
$$
\mathcal D:=\{\mathrm{div}(a\bar{x}+b\bar{y}+c)+E \; \mid \; a,b,c \in \K , \; (a,b,c)\neq (0,0,0)\},
$$
where
$$
E=\sum_{\gamma\in \mathbb{P}_{\mathcal{C}}}e_{\gamma}\gamma ,
     \; \text{with  } e_{\gamma}=-\mathrm{min}\{v_{\gamma}(\bar{x}), v_{\gamma}(\bar{y}),
     v_{\gamma}(1)\} \text{ .}
$$
This set $\mathcal D$ is a linear series, which is usually called the linear series cut out by the lines of $\mathbb P^2(\K)$.
For basic definitions on linear series we refer to \cite{HKT}.
There is a one-to-one correspondence between $\mathcal{D}$
and the set of all lines in $\mathbb P^2(\K)$:  a line $\ell$ with homogeneous equation $aX_0+bX_1+cX_2=0$
corresponds to the divisor
$D(\ell):=\mathrm{div}(a\bar{x}+b\bar{y}+c)+E$. 

For a
place $\gamma$ with $(\mathcal{D},\gamma)$ order sequence
$(0,j_{1}(\gamma),j_{2}(\gamma))$,  and for every line $\ell$, we have $$v_{\gamma}(D(\ell)) \in
\{0,j_{1}(\gamma),j_{2}(\gamma)\}.$$
A line $\ell$ passes through the center of $\gamma$ if and only if $v_\gamma(D(\ell))>0$; also,  there exists a unique line $\ell$ with $v_{\gamma}(D(\ell))=j_{2}(\gamma)$, which is called the {\em tangent line of the place $\gamma$}. The tangent line of a place $\gamma$ is one of the tangent lines of $\cC$ at the center of $\gamma$. The integers $j_1(\gamma)$ and $j_2(\gamma)-j_1(\gamma)$ are called the {\em order} and the {\em class} of $\gamma$, respectively. A place with order equal to $1$ is called a {\em linear} place of $\mathcal C$.

% In fact $v_{P}(D(\ell))$ can be given as
%follows:
%\begin{equation*}
%     v_{P}(D(\ell))= \left\{ \begin{array}{ll}
%                            0        &, \text{if } Q \notin \ell \\
%                            j_{1}(P) & ,\text{if } Q \in \ell \text{ and } \ell \text{ is not the tangent line of the place } P\\
%                            j_{2}(P) & ,\text{if } \ell \text{ is the tangent line of the place } P
%                          \end{array} \right.
%\end{equation*}
%
%
%\vspace{.2cm}

\begin{proposition}\label{val}
Let $Q$ be a point of $\mathcal{C}$ and $\ell$ be a line in
$\mathbb{P}^{2}(\K)$. Then the sum
\begin{equation*}
    \sum_{\begin{array}{c}
          %  \gamma\in \mathbb{P}_{\mathcal{C}} \\
            \gamma \text{ centered at } Q
          \end{array}
    } {v_{\gamma}(D(\ell))} 
    %\;= \;\mathrm{I}(Q,\mathcal{C}\cap \ell).
\end{equation*}
is equal to the intersection multiplicity $\mathrm{I}(Q,\mathcal{C}\cap \ell)$ of $\cC$ and $\ell$ at $Q$. 
\end{proposition}
If $\ell$ is a line through $Q$ which is not a tangent of $\cC$ at $Q$, then $v_\gamma(D(\ell))=j_1(\gamma)$ for each place $\gamma$ centered at $Q$.
Therefore, if $Q$ is an $m$-fold point of $\cC$, then the sum of the orders of the places centered at $Q$ coincides with $m$. 
Also, the number of places centered at $Q$ is greater than or equal to the number of distinct tangents at $Q$.

%\subsection{Bicovering arcs and complete caps in affine spaces}\label{robanota}

\section{A family of curves defined over $\fq$}\label{sec22}
%For $a,b,t \in \K$, with $t\neq 0$, and for a positive integer $m$, we define the %following polynomial:
%\begin{equation}\label{curva2}
%\begin{array}{rcl}
%f_{a,b,t,m}(X,Y)& = &a(t^3X^{2m}Y^m+t^3X^mY^{2m}-3t^2X^{m}Y^m+1)-bt^2X^mY^m 
%\\ & & -t^4X^{2m}Y^{2m}+3t^2X^mY^m-tX^m-tY^m.
%\end{array}
%\end{equation}

Throughtout this section $q=p^h$ for some prime $p>3$, and $m$ is a proper divisor of $q-1$ with $(m,6)=1$. Also, $t$ is a non-zero element in $\fq$ which is not an $m$-th power in $\fq$. For $a,b\in \fq$ with $ab\neq (a-1)^3$, let $P=(a,b)\in AG(2,q)$. A crucial role for the investigation of the bicovering properties of a coset of index $m$ in the abelian group of the non-singular $\fq$-rational points of a nodal cubic is played by the curve
\begin{equation}\label{curva2bis}
\cC_P: f_{a,b,t,m}(X,Y)=0, 
\end{equation}
where
\begin{equation}\label{curva2}
\begin{array}{rcl}
f_{a,b,t,m}(X,Y)& = &a(t^3X^{2m}Y^m+t^3X^mY^{2m}-3t^2X^{m}Y^m+1) 
\\ & & -bt^2X^mY^m-t^4X^{2m}Y^{2m}+3t^2X^mY^m-tX^m-tY^m.
\end{array}
\end{equation}

%arc comprising the points of a coset of index $m$ in the abelian group of the non-%singular $\fq$-rational points of a nodal cubic; see Section \ref{penultima}. 

In \cite{MR1221589,TamasRoma} it is claimed without proof that $\cC_P$ is absolutely irreducible of genus less than or equal to some absolute constant times $m^2$. The proof does not seem to be straightforward. In particular, Segre's  criterion  (\cite{MR0149361}; see also \cite[Lemma 8]{MR0295201}) cannot be applied. 
Actually, for $a^3=-1$ and $b=1-(a-1)^3$, the polynomial $f_{a,b,t,m}(X,Y)$ is reducible; in fact, 
$$
f_{a,b,t,m}(X,Y)=-(a^2+t^2X^mY^m-atY^m)(a^2+t^2X^mY^m-atX^m).
$$
The first result of this section is the existence of an absolutely irreducible component of $\cC_P$ defined over $\fq$. We distinguish a number of cases.

\subsection{$a^3=-1$ and $b=1-(a-1)^3$}\label{caso1}

If both $a^3= -1$ and $b= 1-(a-1)^3$ hold, then  the component of $\cC_P$  with equation $a^2+t^2X^mY^m-atX^m=0$ is a generalized Fermat curve over $\fq$ (see  \cite{FGFermat}). As proven in \cite{FGFermat}, such component is absolutely irreducible with genus less than $m^2$.

\begin{proposition}\label{P38} Assume that $a^3=-1$ and $b=1-(a-1)^3$.
Then the curve $\cC_P$ has an irreducible component defined over $\fq$ of genus less than $m^2$, with equation $a^2+t^2X^mY^m-atX^m=0$.
\end{proposition}

\subsection{$a\neq 0$ and either $a^3\neq -1$ or $b\neq 1-(a-1)^3$}\label{alpha12}
\begin{lemma}\label{9apr} Assume that $ab\neq (a-1)^3$. Then
the plane quartic curve $\cQ_P:g_P(X,Y)=0$ with
$$
\begin{array}{rcl}
g_P(X,Y)& = & a(t^3X^{2}Y+t^3XY^{2}-3t^2XY+1)-bt^2XY 
\\ & & -t^4X^{2}Y^{2}+3t^2XY-tX-tY
\end{array}
$$
is absolutely irreducible.
\end{lemma}
\begin{proof}

Let $X_\infty$ and $Y_\infty$ be the ideal points of the $X$-axis and the $Y$-axis, respectively.
It is straightforward to check that $X_\infty$ and $Y_\infty$ are the only ideal points of $\cQ_P$, and that they are both ordinary double points. The tangent lines of $\cQ_P$ at $X_\infty$ are $Y=0$ and $Y=a/t$; similarly, $X=0$ and $X=a/t$ are the tangent lines at $Y_\infty$. As $ab\neq (a-1)^3$, it is straightforward to check that none of such lines is a component of $\cQ_P$; hence, $\cQ_P$ has no linear component. Assume now that $\cQ_P$  splits into two irreducible conics, say $\cC_1$ and $\cC_2$. 

Without loss of generality we can assume that $X=0$ and $X=a/t$ are the tangents of $\cC_1$ and $\cC_2$ at $Y_\infty$, respectively. 

We first consider the case where $Y=0$ is the tangent of $\cC_1$ at $X_{\infty}$ and $Y=a/t$ is the tangent of $\cC_2$ at $X_{\infty}$. 
Then $\cC_1:XY+\epsilon=0$ and $\cC_2:(X-a/t)(Y-a/t)+ \bar \epsilon=0$ for some $\epsilon, \bar \epsilon \in \K^*$. So for some $\rho\in \K^*$
$$\rho g_P(X,Y)=(XY+\epsilon)((X-a/t)(Y-a/t)+ \bar \epsilon).$$
By comparing coefficients we obtain
$$\left\{ \begin{array}{l} 
-\rho t^4=1\\
%\rho at^3=-\frac{a}{t}\\
%\rho \left(3t^2-3at^2-bt^2 \right)= \frac{a^2}{t^2}+\epsilon+\bar{\epsilon}\\
-\rho t=- \epsilon \frac{a}{t}\\
\rho a=\epsilon\frac{a^2}{t^2}+\epsilon \bar{\epsilon}\\
\end{array} \right. 
\Rightarrow \left\{ \begin{array}{l} 
%\rho=-\frac{1}{t^4}\\
%-\frac{a}{t}=-\frac{a}{t}\\
%-\frac{1}{t^4} \left( 3t^2-3at^2-bt^2 \right)= \frac{a^2}{t^2}+\epsilon+\bar{\epsilon}\\
\frac{1}{t^3}=- \epsilon \frac{a}{t}\\
-\frac{a}{t^4}=\epsilon \frac{a^2}{t^2}+\epsilon \bar{\epsilon}\\
\end{array} \right. \Rightarrow \epsilon \bar \epsilon=0,$$
%$$\Rightarrow \left\{ \begin{array}{l} 
%\rho=-\frac{1}{t^4}\\
%3a+b-3=a^2+\bar{\epsilon} t^2 + \epsilon t^2 \\
%a\epsilon=-\frac{1}{ t^2} \\
%-\frac{a}{t^2}=\epsilon a^2+\epsilon \bar{\epsilon} t^2\\
%\end{array} \right.
%\Rightarrow \bar \epsilon=0,
%\left\{ \begin{array}{l} 
%\rho=-\frac{1}{t^4}\\
%a \epsilon=-\frac{1}{ t^2}\\
%\bar{\epsilon}=0 \\
%3a+b-3=a^2-\frac{1}{a} \\
%\end{array} \right. ,
%$$
which is impossible. 
%$$b+3(a-1)=a^2-\frac{1}{a}$$
%$$ab+3(a-1)a=a^3-1$$
%$$ab+3a^2-3a-a^3+1=0$$
%$$ab+(1-a)^3=0$$
%$$b=-\frac{(1-a)^3}{a}=\frac{(a-1)^3}{a}$$
%This is not possible since the point $P=(a,b) \notin \cX$ by hypothesis.

We now assume that $Y=a/t$ is the tangent of $\cC_1$ at $X_{\infty}$ and $Y=0$ is the tangent of $\cC_2$ at $X_{\infty}$. Then for some $\epsilon,\bar{\epsilon},\rho \in \K^*$
%
%
%In the second case the equations of $\mathcal{C}_{1}$ and $\mathcal{C}_{2}$ are respectively contained in the %pencils of equations $X\left(Y-\frac{a}{t}\right) +\epsilon= 0$ and $\left(X -\frac{a}{t}\right)Y + %\bar{\epsilon}= 0$, where $\epsilon$ and $\bar{\epsilon} \in \mathbb{F}_{q}$. So, up to a factor $\rho \in %\mathbb{F}_{q}$, the equation of the quartic $\mathcal{C}_{P}$ can be described by the product of these two %pencils, namely:

$$\rho g_P(X,Y)=\left(X\left(Y-\frac{a}{t}\right) +\epsilon\right)\left(\left(X -\frac{a}{t}\right)Y +\bar{\epsilon}\right).$$
By comparing coefficients we have
$$\left\{ \begin{array}{l} 
-\rho t^4=1\\
\rho at^3=-\frac{a}{t}\\
\rho \left(3t^2-3at^2-bt^2 \right)= \frac{a^2}{t^2}+\epsilon+\bar{\epsilon}\\
-\rho t=-\epsilon \frac{a}{t}\\
-\rho t=-\bar{\epsilon} \frac{a}{t}\\
\rho a=\epsilon \bar{\epsilon}\\
\end{array} \right. 
\Rightarrow \left\{ \begin{array}{l} 
\rho=-\frac{1}{t^4}\\
%\left( -\frac{a}{t}=-\frac{a}{t} \right)\\
3a+b-3=a^2+\bar{\epsilon} t^2 + \epsilon t^2 \\
\frac{1}{t^3}=- \epsilon \frac{a}{t}\\
\frac{1}{t^3}=- \bar{\epsilon} \frac{a}{t}\\
-\frac{a}{t^4}=\epsilon \bar{\epsilon}\\
\end{array} \right.$$
$$\Rightarrow \left\{ \begin{array}{l} 
%\rho=-\frac{1}{t^4}\\
\epsilon t^2 = \bar \epsilon t^2= -1/a \\
%\bar{\epsilon}=-\frac{1}{a t^2}\\
-\frac{a}{t^4}=\frac{1}{a^2 t^4}\\
3a+b-3=a^2-\frac{1}{a}-\frac{1}{a} \\
\end{array} \right.
\Rightarrow 
\left\{ \begin{array}{l} 
%\rho=-\frac{1}{t^4}\\
%\epsilon=\bar{\epsilon}=-\frac{1}{a t^2} \left( \neq 0 \right)\\
a^3=-1 \\
ab=(a-1)^3-1\\
\end{array} \right., $$
which implies that both 
 $a^3=-1$
 % we obtain
%$$-2+a^3=3a^2+ab-3a \Leftrightarrow ab= -3a^2+3a-2+a^3 \Leftrightarrow ab=(a-1)^3-1 \Leftrightarrow$$
%$$\Leftrightarrow 
and $b=1-(a-1)^3$ hold, a contradiction.

 \end{proof}

Let $\bar u$ and $\bar z$ denote the rational functions of $\K(\cQ_P) $ associated to the affine coordinates $X$ and $Y$, respectively.
Then
\begin{equation}\label{eqff}
a(t^3\bar u^{2}\bar z+t^3\bar u\bar z^{2}-3t^2\bar u \bar z+1)-bt^2\bar u \bar z -t^4\bar u^{2}\bar z^{2}+3t^2\bar u\bar z-t\bar u-t\bar z=0.
\end{equation}
By the proof of Lemma \ref{9apr} both $X_\infty$ and $Y_\infty$ are ordinary double points of $\cQ_P$; hence, they both are the center of two linear places of $\K(\bar u,\bar z)$.
\begin{lemma}\label{L12} Let $\gamma_1$  be the linear  place of $\K(\bar u,\bar z)$ centered at $X_\infty$ with tangent $Y=a/t$, and  
 $\gamma_2$  the linear  place of $\K(\bar u,\bar z)$ centered at $X_\infty$ with tangent $Y=0$. Then
$$
v_{\gamma_1}(\bar u)=-1,\qquad v_{\gamma_1}(\bar z)=0,
$$
and
$$
v_{\gamma_2}(\bar u)=-1,\qquad v_{\gamma_2}(\bar z)>0.
$$
\end{lemma}
\begin{proof}
We keep the notation of Section \ref{order}. Here, the role of $\bar x$ and $\bar y$  is played by $\bar u$ and $\bar z$, respectively. Then
%Note that $\bar u$   corresponds to the line $X=0$ not passing through $X_\infty$, whereas $\bar z$ to the 
%tangent line $Y=0$ through the place $\gamma$. Then,
\begin{equation}\label{eq3}
 v_{\gamma_1}(\bar{z}-a/t )+e_{\gamma_1}=j_2(\gamma_1),
\end{equation}
\begin{equation}\label{eq2}
 v_{\gamma_1}\left(\bar{u}\right)+e_{\gamma_1}=0,
\end{equation}
\begin{equation}\label{eq4}
 v_{\gamma_1}(\bar{z})+e_{\gamma_1}=1.
\end{equation}
From here one can easily deduce that $v_{\gamma_1}(\bar z)=0$. In fact, 
if $v_{\gamma_1}(\bar z)>0$, then $v_{\gamma_1}(\bar z - a/t)=0$, and hence $e_{\gamma_1}=j_2(\gamma_1)$; also, \eqref{eq4} implies  $j_2(\gamma_1)=1$, a contradiction. On the other hand,
 if $v_{\gamma_1}(\bar z)<0$, then $v_{\gamma_1}(\bar z - a/t)=v_{\gamma_1}(\bar z)$; hence, \eqref{eq3} and \eqref{eq4} yield that $j_2(\gamma_1)=1$, a contradiction.
 From \eqref{eq4} it follows that $e_{\gamma_1}=1$; then   $v_{\gamma_1}(\bar u)=-1$ is obtained from \eqref{eq2}. 

As far as $\gamma_2$ is concerned, note that 
\begin{equation}\label{eq5}
 v_{\gamma_2}(\bar{z}-a/t )+e_{\gamma_2}=1,
\end{equation}
\begin{equation}\label{eq6}
 v_{\gamma_2}\left(\bar{u}\right)+e_{\gamma_2}=0,
\end{equation}
\begin{equation}\label{eq7}
 v_{\gamma_2}(\bar{z})+e_{\gamma_2}=j_2(\gamma_2).
\end{equation}
Then the assertion about $\gamma_2$ can be easily obtained from $j_2(\gamma_2)>1$.
 \end{proof}

As $\cQ_P$ is left invariant by the transformation $X\mapsto Y$, $Y\mapsto X$, the following result is obtained at once.
\begin{lemma}\label{L13} Let $\gamma_3$  be the linear  place of $\K(\bar u,\bar z)$ centered at $Y_\infty$ with tangent $X=a/t$, and  
 $\gamma_4$  the linear  place of $\K(\bar u,\bar z)$ centered at $Y_\infty$ with tangent $X=0$. Then
$$
v_{\gamma_3}(\bar u)=0,\qquad v_{\gamma_3}(\bar z)=-1,
$$
and
$$
v_{\gamma_4}(\bar u)>0,\qquad v_{\gamma_4}(\bar z)=-1.
$$
\end{lemma}
Let $Q_1=(0,a/t)$ and $Q_2=(a/t,0)$. It is easily seen that both $Q_1$ and $Q_2$ are simple points of $\cQ_P$, and hence they both are the center of precisely one linear place of $\K(\bar u,\bar z)$

\begin{lemma}\label{L14} Let $\gamma_5$ be the place of $\K(\bar u,\bar z)$ centered at $Q_1$, and $\gamma_6$ the place centered at $Q_2$. Then
$$
\div(\bar u)=\gamma_4+\gamma_5-\gamma_1-\gamma_2,
$$
and 
$$
\div(\bar z)=\gamma_2+\gamma_6-\gamma_3-\gamma_4.
$$
\end{lemma}
\begin{proof}
Clearly, $\gamma_5$ is a zero of $\bar u$, whereas $\gamma_6$ is a zero of $\bar z$. From \eqref{eqff}, the number of zeros (and poles) of either $\bar u$ or $\bar z$ is $2$. Then the assertion follows from Lemmas \ref{L12} and \ref{L13}.
 \end{proof}

We now consider the extension
 $\K(\bar u,\bar z)(\bar y)$ of
$\K(\bar u,\bar z)$ defined by the equation $\bar y^m=\bar z$. Clearly, $\K(\bar u,\bar z, \bar y)=\K(\bar u, \bar y)$ holds.
By Lemma \ref{L14}, $\K(\bar u,\bar y)$ is a Kummer extension of $\K(\bar u,\bar z)$. For a place $\gamma$ of $\K(\bar u,\bar z)$ let $r_\gamma=\gcd(m,v_\gamma(\bar z))$. Then by Lemma \ref{L14} we have
$$
\left\{
\begin{array}{ll}
r_{\gamma}=1, & \text{ if }\gamma \in \{\gamma_2,\gamma_3,\gamma_4,\gamma_6\},\\
r_\gamma=m, & \text{otherwise}.
\end{array}
\right.
$$
By Proposition \ref{teo1} the genus of $\K(\bar u,\bar y)$ is equal to $2m-1+m(g-1)$, where $g$ denotes the genus of $\cQ_P$. Since $\cQ_P$ is a quartic with two double points, $g\le 1$ holds and hence the genus of 
$\K(\bar u,\bar z,\bar y)$ is less than or equal to $2m-1$.
Also, the  places of $\K(\bar u, \bar z)$ which ramify in the extension $\K(\bar u,\bar y):K(\bar u, \bar z)$ are precisely $\gamma_2,\gamma_3,\gamma_4,\gamma_6$; their ramification index is $m$.
For $i\in \{2,3,4,6\}$ let $\bar \gamma_i$ be the only place of $\K(\bar u, \bar y)$ lying over $\gamma_i$; also, let $\bar \gamma_1^1,\ldots, \bar \gamma_1^m$ be the places of $\K(\bar u,\bar y)$ lying over $\gamma_1$ and let $\bar \gamma_5^1,\ldots, \bar \gamma_5^m$ be the places of $\K(\bar u, \bar y)$ lying over $\gamma_5$.
Taking into account Lemma \ref{L14}, the divisor of $\bar u$ in $\K(\bar u, \bar y)$ can be easily computed.

\begin{lemma}\label{L145} In $\K(\bar u, \bar y)$,
$$
\div(\bar u)=m\bar \gamma_4+\sum_{i=1}^m \bar \gamma_5^i  -m\bar \gamma_2 -\sum_{i=1}^m \bar \gamma_1^i.
$$
\end{lemma}
We can now apply Proposition \ref{teo1}, together with Lemma \ref{L145}, in order to  deduce that the extension $\K(\bar u, \bar y)(\bar x)=\K(\bar y,\bar x)$ of
$\K(\bar u, \bar y)$ defined by the equation $\bar x^m=\bar u$ is a Kummer extension of $\K(\bar u,\bar y)$ of genus 
$$
1+m\Big(g'-1+\frac{1}{2}\big(1-\frac{1}{m}\big)2m\Big),
$$
where $g'$ is the genus of $\K(\bar u,  \bar y)$. Taking into account that $g'\le 2m-1$, the following result is obtained.
\begin{lemma}\label{L15} The genus of $\K(\bar x,\bar y)$ is at most $3m^2-3m+1$.
\end{lemma}

\begin{proposition}\label{P10} Assume that $a\neq 0$ and either $a^3\neq -1$ or $b\neq 1-(a-1)^3$. 
Then the curve $\cC_P$ is an absolutely irreducible curve defined over $\fq$ with genus less than or equal to $ 3m^2-3m+1$.
\end{proposition}
\begin{proof}
Suppose  that $f_{a,b,t,m}(X,Y)$ admits a non-trivial factorization
$$
f_{a,b,t,m}(X,Y)=g_1(X,Y)^{m_1} \cdots g_s(X,Y)^{m_s},
$$
By construction,  $f_{a,b,t,m}(\bar x,\bar y)=0$ holds and hence 
 there exists $i_0 \in \{1, \ldots, s\}$ such that $g_{i_0}(\bar x, \bar y)=0$.
Clearly, either  $\deg_{X}(g_{i_0})<2m$ or  $\deg_{Y}(g_{i_0})<2m$
 holds. To get a contradiction, it is then enough to show
 that the extensions $\K(\bar x,\bar y):\K(\bar x)$ and
$\K(\bar x,\bar y):\K(\bar y)$ have both degree $2m$. 

From the diagram 

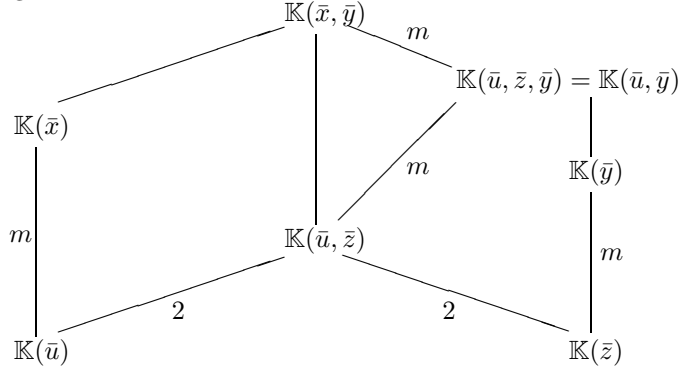
\begin{center}

\setlength{\unitlength}{.6cm}
\begin{picture}(15,10)(0,0)

%Primo blocco

\put(7,10){$\K(\bar{x},\bar{y})$}

%\put(4,9.3){\textcolor{red}{\framebox{$2m$}}}

\put(9.7,9.6){$m$}
% \qquad (\bar{u}^{m}=\bar{x})$}

\put(9.7,6.6){$m$}

\put(0.9,5.1){$m$}

%\put(7,7.5){$m^2$}

\put(4.5,3.4){$2$}

\put(10.5,3.4){$2$}

\put(14,4.7){$m$}
%\qquad (\bar{z}^m=\bar{y})$}

%\put(14,7.5){\textcolor{red}{\framebox{$2$}}}

\put(2,8.2){\line(3,1){5}}

\put(7,5){$\K(\bar{u},\bar{z})$}

\put(1.5,7.2){\line(0,-1){4.2}}

\put(7.7,9.7){\line(0,-1){4.2}}

\put(1,7.5){$\K(\bar{x})$}

\put(8.4,9.9){\line(5,-2){2.3}}

\put(8.2,5.6){\line(1,1){2.6}}   

\put(10.8,8.5){$\K(\bar{u},\bar{z},\bar{y})=\K(\bar{u},\bar{y})$}

\put(13.3,6.5){$\K(\bar{y})$}

\put(13.8,8.3){\line(0,-1){1.3}}

\put(1,2.5){$\K(\bar{u})$}

\put(13.8,6.2){\line(0,-1){3.1}}

\put(2,3.1){\line(3,1){5}}

\put(8.3,4.8){\line(3,-1){5}}

\put(13.3,2.5){$\K(\bar{z})$}

\end{picture}
\end{center}
it follows that $[\K(\bar x,\bar y):\K(\bar u)]=[\K(\bar x,\bar y):\K(\bar z)]=2m^2$; hence both $[\K(\bar x,\bar y):\K(\bar y)]=2m$ and $[\K(\bar x,\bar y):\K(\bar x)]=2m$ hold.

%As clearly $\fq(\cQ_P)(\bar y)(\bar x)=\fq(\bar x, \bar y)$, with $f(\bar x, \bar y)=0$, it follows that $\cC_P: f_P(X,Y)=0$ is an absolutely irreducible curve. 

Then $\K(\bar x,\bar y)$ is the function field of $\cC_P$, and the assertion on the genus follows from Lemma \ref{L15}.
 \end{proof}
\subsection{$a=0$}\label{caso3}

\begin{lemma}\label{L16}
The plane quartic curve $\cQ_P$ with equation
$$
-bt^2XY -t^4X^{2}Y^{2}+3t^2XY-tX-tY=0
$$
is absolutely irreducible of genus $g\le 1$.
\end{lemma}
\begin{proof}
It is easily seen that $\cQ_P$ does not admit any linear component. Note that 
both $X_\infty$ and $Y_\infty$ are cuspidal double points of $\cQ_P$. The tangent line at $X_\infty$ is $Y=0$, and the intersection multiplicity of $\cQ_P$ and $Y=0$ at $X_\infty$ is equal to $3$; similarly, $X=0$ is the tangent line at $Y_\infty$ and $\mathrm I(Y_\infty, \cQ_P\cap \{X=0\})=3$. Therefore,  precisely one irreducible component $\cC$ of $\cQ_P$ passes through $Y_\infty$; also,  $Y_\infty$ is a double point of $\cC$ and there is only one place of $\K(\cC)$ centered at $Y_\infty$. Then $\cC$ is a curve of degree greater than $2$. Since $\cQ_P$ does not have any linear component, the only possibility is that the degree of $\cC$ is four, that is, $\cC=\cQ_P$. This shows that 
$\cQ_P$ is absolutely irreducible. As $\cQ_P$ is a quartic with at least two singular points, its genus $g$ is less than or equal to $1$.
 \end{proof}
Let $\K(\bar u,\bar z)$ be the function field of $\cQ_P$. Here, $\bar u$ and $\bar z$ are rational functions on $\cQ_P$ such that 
$$
-bt^2\bar u \bar z -t^4\bar u^{2}\bar z^{2}+3t^2\bar u \bar z-t\bar u-t\bar z=0.
$$
Let $\gamma_1$ be the only place of $\K(\bar u,\bar z)$ centered at the (simple) point of $\cQ_P$ with coordinates $(0,0)$. From the proof of Lemma \ref{L16} there is precisely one place of $\K(\bar u,\bar z)$, say $\gamma_2$, centered at $Y_\infty$. As $\cQ_P$ is left invariant by the transformation $X\mapsto Y$, $Y\mapsto X$, the same holds for $X_\infty$; we denote by $\gamma_3$ the only place of $\K(\bar u,\bar z)$ centered at $X_\infty$. 
Arguing as in the proofs of Lemmas \ref{L12}, \ref{L13} and \ref{L14}, the divisors of both $\bar u$ and $\bar z$ can be computed.

\begin{lemma}\label{L17}
In $\K(\bar u, \bar z)$,
$$
\div(\bar u)=\gamma_1+\gamma_2-2\gamma_3,\qquad \div(\bar z)=\gamma_1+\gamma_3-2\gamma_2.
$$
\end{lemma}

In order to prove that $\cC_P$ is absolutely irreducible, the same arguments as in Section \ref{alpha12} can be used.
Let
 $\K(\bar u,\bar z)(\bar y)$ be the extension of
$\K(\bar u,\bar z)$ defined by the equation $\bar y^m=\bar z$. Clearly, $\K(\bar u,\bar z, \bar y)=\K(\bar u, \bar y)$ holds.
By Lemma \ref{L17} $\K(\bar u,\bar y)$ is a Kummer extension of $\K(\bar u,\bar z)$. %For a place $\gamma$ of $\K(\bar u,\bar z)$ let $r_\gamma=\gcd(m,v_\gamma(\bar z))$. 
As $m$ is odd, by Lemma \ref{L17} we have that 
%if $m$ is odd, then
$$
\left\{
\begin{array}{ll}
r_{\gamma}=1, & \text{ if }\gamma \in \{\gamma_1,\gamma_2,\gamma_3\},\\
r_\gamma=m, & \text{otherwise}.
\end{array}
\right.
$$
%if $m$ is even, then
%$$
%\left\{
%\begin{array}{ll}
%r_{\gamma}=1, & \text{ if }\gamma \in \{\gamma_1,\gamma_3\}\\
%r_{\gamma}=2, & \text{ if }\gamma=\gamma_2 \\
%r_\gamma=m, & \text{otherwise}
%\end{array}
%\right..
%$$
By Proposition \ref{teo1} the genus  of $\K(\bar u,\bar y)$ is equal to 
\begin{equation}\label{gen21}
g'=
%\left\{
%\begin{array}{ll}
m(g-1)+\frac{3m-1}{2}, 
%& \text{ if } m \text{ is odd}\\
%m(g-1)+\frac{3m-2}{2} & \text{ if } m \text{ is even} 
%\end{array}
%\right.,
\end{equation}
where $g\in \{0,1\}$ denotes the genus of $\cQ_P$. 
Also, the  places of $\K(\bar u, \bar z)$ which ramify in the extension $\K(\bar u,\bar y):\K(\bar u, \bar z)$ are precisely $\gamma_1,\gamma_2,\gamma_3$; their ramification index is $m$.
% with the exception of $\gamma_2$ for $m$ even, where the ramification index is $m/2$.
For  $i\in \{1,2,3\}$ let $\bar \gamma_i$ be the only place of $\K(\bar u, \bar y)$ lying over $\gamma_i$.
%, and
%. When $m$ is odd, 
%let  $\bar \gamma_2$ be the only place of $\K(\bar u, \bar y)$ lying over $\gamma_2$.
%; for $m$ even, 
% let $\bar \gamma_2^1, \bar \gamma_2^2$ be the places of $\K(\bar u,\bar y)$ lying over %$\gamma_2$. 
Taking into account Lemma \ref{L17}, the divisors of both $\bar u$ and $\bar y$ in $\K(\bar u, \bar y)$ can be easily computed.

\begin{lemma} In $\K(\bar u, \bar y)$, 
%if $m$ is odd then
$$
\div(\bar u)=m\bar \gamma_1+m\bar \gamma_2-2m\bar \gamma_3,\qquad 
\div(\bar y)= \bar \gamma_1+\bar \gamma_3-2\bar \gamma_2. 
$$
%if $m$ is even then
%$$
%\div(\bar u)=m\bar \gamma_1+\frac{m}{2} \bar \gamma_2^1+\frac{m}{2}\bar \gamma_2^2-2m\bar %\gamma_3,\qquad 
%\div(\bar y)= \bar\gamma_1+\bar \gamma_3-\bar \gamma_2^1-\bar \gamma_2^2. 
%$$
\end{lemma}
We now consider the extension $\K(\bar u, \bar y)(\bar x)=\K(\bar y,\bar x)$ of
$\K(\bar u, \bar y)$ such that $\bar x^m=\bar u$.
In order to apply Proposition \ref{teo1}, we need to determine whether the rational function $\bar u$ is an $e$-th power in $\K(\bar u,\bar y)$, for some divisor $e$ of $m$.

\begin{lemma}\label{P21} 
%Assume that for some divisor $h>1$ of $m$ 
The rational function $\bar u$ is not an $e$-th power in $\K(\bar u,\bar y)$ for any divisor $e>1$ of $m$. 
%Then $h\in \{2,3\}$ and $g'=\lfloor \frac{m-1}{2}\rfloor$. 
\end{lemma}
\begin{proof}
%We first deal with the case $m$ odd.
Assume that $\bar u=\bar v^e$, with $e$ a non-trivial divisor of $m$. 
%As $m$ is odd, $h\ge 3$ holds. 
Then
$$
\div(\bar v)=\frac{m}{e}\bar \gamma_1+\frac{m}{e}\bar \gamma_2-\frac{2m}{e}\bar \gamma_3,\qquad 
$$
Consider the rational function $\bar v\bar y^{i}$ for $-\frac{m}{e}\le i \le \big(\frac{m}{e}-1\big)/2$. The pole divisor of $\bar v\bar y^{i}$ is $\big(\frac{2m}{e}-i\big)\bar \gamma_3$, which shows that
%For each integer $i=0,\ldots, \big(\frac{m}{h}-1\big)/2$, the pole divisor of $\bar v\bar %y^{i}$ is $\big(\frac{2m}{h}-i\big)\bar \gamma_3$. 
%Also, for each $i=1,\ldots, \frac{m}{h}$, the pole divisor of $\frac{\bar v}{\bar y^{i}}$ is %$\big(\frac{2m}{h}+i\big)\bar \gamma_3$.
%This shows that 
the Weierstrass semigroup $H(\bar \gamma_3)$ at $\bar \gamma_3$ contains
$$
\frac{3m}{2e}+\frac{1}{2},\frac{3m}{2e}+\frac{3}{2},\ldots, \frac{3m}{e},
$$
and hence every integer greater than or equal to $\frac{3m}{2e}+\frac{1}{2}$.
As $g'$ is equal to the number of gaps in $H(\bar \gamma_3)$ we have
$$
g'\le \frac{3m}{2e}-\frac{1}{2};
$$
by \eqref{gen21} this can only happen 
 when both $e=3$ and $g'=(m-1)/2$ hold. But this is impossible as $(m,6)=1$ is assumed.
 \begin{comment}
%\subsubsection{Case $g=0$, $m$ odd}
%Si ragiona come sopra; cambia il genere $g'$ che \`e $(m-1)/2$. Per avere la contraddizione allora serve che 
%$\frac{3m}{h}+1$ SIA STRETTAMENTE MINORE DI $\frac{m-1}{2}$, ovvero $h>\frac{6m}{m-3}$, cio\`e $h\ge 7$. {\bf Mi %serve che $m$ non abbia divisori pi\`u piccoli di $7$}.
%
%SE INTENDIAMO PROSEGUIRE SU QUESTA LINEA IL CASO $m$ PARI SI PU\'O BUTTARE A MARE.

For $m$ even, we argue as in the case $m$ odd.
Assume that $\bar u=\bar v^h$, with $h$ a non-trivial divisor of $m$.  Then
$$
\div(\bar v)=\frac{m}{h}\bar \gamma_1+\frac{m}{2h}\bar \gamma_2^1+\frac{m}{2h}\bar \gamma_2^2-\frac{2m}{h}\bar \gamma_3,\qquad 
$$
For each integer $i=-\frac{m}{h},\ldots, \frac{m}{2h}$, the pole divisor of $\bar v\bar y^{i}$ is $\big(\frac{2m}{h}-i\big)\bar \gamma_3$. 
Then Weierstrass semigroup at $\bar \gamma_3$ contains
$$
\frac{3m}{2h},\frac{3m}{2h}+1,\ldots, \frac{3m}{h},
$$
whence
$$
g'\le \frac{3m}{2h}-1;
$$
again, this is possible only if both $h\in\{2,3\}$ and $g'=\frac{m}{2}-1$ hold.
\end{comment}
 \end{proof}
Arguing as in the proofs of Lemma \ref{L15} and Proposition \ref{P10}, the following result is obtained.
\begin{comment}
\begin{proposition}\label{P22} If $\bar u$ is not an $h$-th power in $\K(\bar u, \bar y)$ for any divisor $h>1$ of $m$, then the curve $\cC_P:f_{a,b,t,m}(X,Y)=0$ is an absolutely irreducible curve defined over $\fq$ with genus less than or equal to $\frac{3m^2-3m+2}{2}.
$
\end{proposition}
\end{comment}
\begin{proposition} Assume that $a=0$. Then the curve $\cC_P$ is an absolutely irreducible curve defined over $\fq$ with genus less than or equal to 
$
\frac{3m^2-3m+2}{2}.
$
\end{proposition}

\subsection{Some double covers of $\cC_P$}
%In this section and throughout the rest of the paper we assume that $m$ is odd.
In the three-dimensional space over $\K$, fix an affine coordinate system $(X,Y,W)$ and for any $c\in \K$, $c\neq 0$ let
$\cY_P$ be the curve defined by
$$\cY_P:
\left\{
\begin{array}{l}
W^2=c(a-tX^m)(a-tY^m)\\
f_{a,b,t,m}(X,Y)=0
\end{array}
\right..
$$
The existence of a suitable $\fq$-rational point of $\cY_P$ will guarantee that $P$ is  bicovered by the arc comprising the points of a coset of index $m$ in the abelian group of the non-singular $\fq$-rational points of a nodal cubic; see Section \ref{penultima}. 

\begin{proposition}\label{exiexi} Let $a,b\in \fq$ be such that $ab\neq (a-1)^3$. For each $c\in \fq$, $c\neq 0$, the space curve $\cY_P$ 
has an irreducible component defined over $\fq$ with genus less than or equal to $6m^2-4m+1$.
\end{proposition}
\begin{proof}
We distinguish a number of cases.

{\bf Case 1:} {$a^3=-1$ and $b= 1-(a-1)^3$}.

Notation here is as in Section \ref{caso1}.
The function field of an $\fq$-rational irreducible component $\cC$ of $\cC_P$ is $\K(\bar x, \bar y)$ with
$$
a^2+t^2\bar x^m\bar y^m-at\bar x^m=0.
$$
By the results on generalized Fermat curves presented in \cite{FGFermat}, the genus of $\cC$ is $(m^2-3m+2)/2$; also there are $m$ places, say $\gamma_1^1,\ldots,\gamma_1^m$ of $\K(\bar x,\bar y)$ centered at $X_\infty$, and $m$ places, say $\gamma_2^1,\ldots,\gamma_2^m$ of $\K(\bar x,\bar y)$ centered at $Y_\infty$. Let $\gamma_3^1,\ldots,\gamma_3^m$ denote the places centered at the $m$ simple affine points of $\cC$ with coordinates $(v,0)$ with $v^m=a/t$. 
%From the results in Section 3.1 
We have
\begin{eqnarray*}
\div(\bar x)=\gamma_2^1+\ldots+\gamma_2^m-(\gamma_1^1+\ldots+\gamma_1^m),\\
\div(\bar y)=\gamma_3^1+\ldots+\gamma_3^m-(\gamma_2^1+\ldots+\gamma_2^m).
\end{eqnarray*}
Then it is easy to see that
\begin{eqnarray*}
\div(a-t\bar x^m )=m(\gamma_3^1+\ldots+\gamma_3^m)-m(\gamma_1^1+\ldots+\gamma_1^m),\\ 
\div(a-t\bar y^m)=m(\gamma_1^1+\ldots+\gamma_1^m) -m(\gamma_2^1+\ldots+\gamma_2^m),
\end{eqnarray*}
whence 
$$
\div\big( (a-t\bar x^m )(a-t\bar y^m)\big)=m(\gamma_3^1+\ldots+\gamma_3^m)-m(\gamma_2^1+\ldots+\gamma_2^m).
$$
As $m$ is odd, this yields that $(a-t\bar x^m)(a-t\bar y^m)$
is not a square in $\K(\bar x, \bar y)$.
By Proposition \ref{teo1cor}, for each $c\in \fq$, $c\neq 0$, the space curve with equations
$$
\left\{
\begin{array}{l}
W^2=c(a-tX^m)(a-tY^m)\\
a^2+t^2X^mY^m-atX^m=0
\end{array}
\right.
$$ 
has an irreducible component defined over $\fq$ with genus  $m^2-2m+1$.
The claim then follows as such curve is contained in $\cY_P$ as well.

{\bf Case 2:} {$a\neq 0$ and either $a^3\neq - 1$ or $b\neq 1-(a-1)^3$.}

We keep the notation of Section \ref{alpha12}. By Lemma \ref{L145}, the only 
 places of $\K(\bar u, \bar y)$ which ramify in the extension $\K(\bar x,\bar y):K(\bar u, \bar y)$ are $\bar \gamma_1^1,\ldots,\bar \gamma_1^m$ and $\bar \gamma_5^1,\ldots,\bar \gamma_5^m$; their common ramification index is $m$. Therefore, for each $j=1,\ldots, 6$, the ramification index of $\gamma_j$ in the extension $\K(\bar x,\bar y)$ over $\K(\bar u,\bar z)$ is equal to $m$, and no other place of $\K(\bar u,\bar z)$ is ramified.
%For $j\in \{1,5\}$, $i=1,\ldots, m$, let ${\bar {\bar{\gamma}}}_j^i$ denote the only %place
%of $\K(\bar x, \bar y)$ lying over $\bar \gamma_j^i$. Also, for $j\in \{2,4\}$, let  %${\bar {\bar{\gamma}}}_j^1, \ldots,{\bar {\bar{\gamma}}}_j^m$ denote the  places
%of $\K(\bar x, \bar y)$ lying over $\bar \gamma_j$.
For $j=1,\ldots,6$, let  ${\bar {\bar{\gamma}}}_j^1, \ldots,{\bar {\bar{\gamma}}}_j^m$ denote the  places of $\K(\bar x, \bar y)$ lying over the place $\gamma_j$ of $\K(\bar u, \bar z)$.

From Equations \eqref{eq3}--\eqref{eq7}, together with Lemma \ref{L14}, we deduce that in $\K(\bar u,\bar z)$
$$
\div(a-t\bar z)=\div(\bar z-a/t)=\gamma_1+\gamma_5-\gamma_3-\gamma_4
$$
holds; similarly,  
$$
\div(a-t\bar u)=\div(\bar u-a/t)=\gamma_3+\gamma_6-\gamma_1-\gamma_2.
$$
%Hence, in $\K(\bar u,\bar z)$
%$$
%\div\big((a-t\bar u)(a-t\bar z)\big)=\gamma_5+\gamma_6-\gamma_4-\gamma_2,
%$$
This implies that in $\K(\bar x,\bar y)$
\begin{equation}\label{23genn}
\div\big((a-t\bar x^m)(a-t\bar y^m)\big)=m\Big(\sum_{i=1}^m({\bar{\bar {\gamma}}}_5^i+{\bar{\bar {\gamma}}}_6^i-{\bar{\bar {\gamma}}}_4^i-{\bar{\bar {\gamma}}}_2^i)\Big)
\end{equation}
holds. As $m$ is odd, this yields that $(a-t\bar x^m)(a-t\bar y^m)$
is not a square in $\K(\bar x, \bar y)$.
By Proposition \ref{teo1cor} for each $c\in \fq$, $c\neq 0$, the curve $\cY_P$ has an irreducible component defined over $\fq$ with genus at most $6m^2-4m+1$.

{\bf Case 3:} {$a=0$}
We keep the notation of Section \ref{caso3}. 
%By Lemma \ref{P21}  $\bar u$ is not an $e$-th power in $\K(\bar u, \bar y)$ for any %divisor $e>1$ of $m$. 
The curve $\cC_P$ is absolutely irreducible, and for each $i\in \{1,2,3\}$ the ramification index of $\gamma_i$ in the extension $\K(\bar x,\bar y)$ over $\K(\bar u, \bar z)$ is equal to $m$. By Lemma \ref{L17}  the divisor of $\bar u\bar z$ in $\K(\bar u, \bar z)$ is $2\gamma_1-\gamma_2-\gamma_3$. Hence, in $\K(\bar x, \bar y)$, the rational function $t^2\bar x^m\bar y^m=t^2\bar u \bar z$ has $m$ zeros with multiplicity $2m$ (the places of $\K(\bar x,\bar y)$ lying over $\gamma_1$) and $2m$ poles with multiplicity $m$ (the places lying over $\gamma_2$ and $\gamma_3$). As $m$ is odd and $a=0$, this yields that $(a-t\bar x^m)(a-t\bar y^m)$ is not a square in $\K(\bar x, \bar y)$. Also,
by Proposition \ref{teo1cor}, for each $c\in \fq$, $c\neq 0$, the curve $\cY_P$
has an irreducible component defined over $\fq$ with genus at most $3m^2-2m+1$.

\begin{comment}
If $\bar u$ is an $h$-th power in  $\K(\bar u, \bar y)$ for some divisor $h$ of $m$, then
from the proof of Proposition \ref{P21} $3$ divides $m$ and $h=3$; also, by Proposition \ref{Prop15}, the curve $\cC_P$  has an $\fq$-rational component $\cC_{i_0}:g_{i_0}(X,Y)=0$ of genus $(m^2-3m+6)/6$. For each $i\in \{1,2,3\}$ the ramification index of $\gamma_i$ in the extension $\K(\bar x,\bar y)$ over $\K(\bar u, \bar z)$ is equal to $m$.  
%By Lemma \ref{L17}, 
In $\K(\cC_{i_0})=\K(\bar x, \bar y)$ the rational function $t^2\bar x^m\bar y^m=t^2\bar u \bar z$ has $m/3$ zeros with multiplicity $2m$ (the places of $\K(\bar x,\bar y)$ lying over $\gamma_1$) and $2m/3$ poles with multiplicity $m$ (the places lying over $\gamma_2$ and $\gamma_3$). Then $t^2\bar x^m \bar y^m$ is not a square in $\K(\bar x, \bar y)$ and by Proposition \ref{teo1cor} for each $c\in \fq$, $c\neq 0$, the space curve with equations
$$
\left\{
\begin{array}{l}
W^2=ct^2X^mY^m\\
g_{i_0}(X,Y)=0
\end{array}
\right.
$$
has an irreducible component defined over $\fq$ with genus $(m^2-2m+3)/3$.
\end{comment}
 \end{proof}

\section{Bicovering arcs from nodal cubics}\label{penultima}

If $\cX$ is a singular plane cubic defined over $\fq$ with a node and at least one $\fq$-rational inflection, then a canonical equation for $\cX$ is  $XY=(X-1)^3$. If the neutral element of $(G,\oplus)$ is chosen to be the affine point $(1,0)$, then $(G,\oplus)$ is isomorphic to $(\fq^*,\cdot)$ via the map $v\mapsto (v,(v-1)^3/v)$.

Let $K$ be the subgroup of $G$ of index $m$ with $(m,6)=1$, and let $P_t=(t,(t-1)^3/t)$ be a point in $G\setminus K$. Then the coset $K_t=K\oplus P_t$ is an arc.
In order to investigate  the bicovering properties of the arc $K_t$ it is useful write $K_t$ in an algebraically parametrized form:
\begin{equation}\label{9giu}
K_t=\Big\{\big(tw^m, \frac{(tw^m-1)^3}{tw^m}\big)\mid w \in \fq^*\Big\}.
\end{equation}
%for some $t\in \fq^*$, $t$ not an $m$-th power in $\fq^*$.
For a point $P=(a,b)$ in $AG(2,q)\setminus \cX$, let $f_{a,b,t,m}(X,Y)$ be as in \eqref{curva2}.

\begin{proposition}\label{allinea}
An affine point $P=(a,b)$ in $AG(2,q)\setminus \cX$ is collinear with two distinct points in $K_t$ if and only if there exist $\tilde x,\tilde y\in \fq^*$ with $\tilde x^m\neq \tilde y^m$ such that
$f_{a,b,t,m}(\tilde x,\tilde y)=0$.
\end{proposition}
\begin{proof}
 Two distinct points $P_1=\left(v_1,\frac{(v_1-1)^3}{v_1} \right)$, $P_2=\left(v_2,\frac{(v_2-1)^3}{v_2} \right)$ in $\cX$  are collinear with $P$ if and only if
$$
\left| \begin{array}{ccc}
v_1 & \frac{(v_1-1)^3}{v_1} & 1\\
v_2 & \frac{(v_2-1)^3}{v_2} & 1\\
a & b & 1
\end{array}\right|=0.
$$
As $P_1\neq P_2$, this is equivalent to
%$$
%a(v_1^2 v_2+v_1 v_2^2 - 3v_1 v_2+1)-bv_1 v_2 - v_1^2 v_2^2 + 3 v_1 v_2 -v_1 -v_2=0
%$$
$$
a(v_1^2 v_2+v_1 v_2^2 - 3v_1 v_2+1)-bv_1 v_2 - v_1 v_2(v_1 v_2 -3)-(v_1 +v_2)=0.
$$
When $P_1,P_2$ are elements of $K_t$, both  $v_1=t\tilde x^m$ and $v_2=t\tilde y^m$ hold for some $\tilde x,\tilde y\in \fq^*$, whence the assertion.
%$$
%v_1v_2 \left[b(v_1+v_2)-3b-a-v_1v_2+3\right]-(v_1+v_2)+b
%$$
 \end{proof}

%\section{Points in $AG(2,q)\setminus \cX$ are bicovered by $K_t$ if $m$ is odd}

%Throughout this section we keep on assuming that $(m,6)=1$. 

%For a point $P=(a,b)$ in $AG(2,q)\setminus \cX$, let $\cC_P$ be as in Section %\eqref{curva2bis}.

%As in Section \ref{sec22}, a number of cases need to be distinguished.

\begin{proposition}\label{contiamo} If
\begin{equation}\label{aritmetica1}
q+1-(12m^2-8m+2)\sqrt q \ge 8m^2+8m+1
\end{equation}
then every point $P$ in $AG(2,q)$ off $\cX$ is bicovered by $K_t$.
\end{proposition}
\begin{proof}
We only deal with the case where $a\neq 0$ and either $a^3\neq -1$ or $b\neq 1-(a-1)^3$, the proof for the other cases being analogous. Fix a non-zero element $c$ in $\fq$ and let $\cY_P$ be as in Proposition \ref{exiexi}. Let $\K(\bar x, \bar y, \bar w)$ be the function field of $\cY_P$, so that
$$
\left\{
\begin{array}{l}
\bar w^2=c(a-t\bar x^m)(a-t\bar y^m)\\
f_{a,b,t,m}(\bar x,\bar y)=0
\end{array}
\right.
$$
We argue as in the proof of  Theorem 4.4 in \cite{ABGP}. 
Let $E$ be the set of places $\gamma$ of $\K(\bar x,\bar y,\bar w)$ for which at least one of the following holds:
\begin{enumerate}
\item[(1)] $\gamma$ is either a zero or a pole of $\bar x$;
\item[(2)] $\gamma$ is either a zero or a pole of $\bar y$;
\item[(3)] $\gamma$ is either a zero or a pole of $\bar w$;
\item[(4)] $\gamma$ is a zero of $\bar x^m-\bar y^m$.
\end{enumerate}
%By the proof of \cite[Theorem 5.1]{HV} (see also \cite[Lemma 3.2]{GAAECC}), we have an upper bound on the size of $E$.
%\begin{lemma}\label{count} 
We are going to show that the size of $E$ is at most $8m^2+8m$.
%\end{lemma}
It has already been noticed in the proof of Proposition \ref{exiexi}, Case 2, that the only places of $\K(\bar u,\bar z)$ that ramifies in $\K(\bar x,\bar y)$ are the places $\gamma_j$ for $j=1,\ldots,6$, and their common ramification index is $m$.
Also, by \eqref{23genn}, the degree-$2$ extension $\K(\bar x,\bar y,\bar w)$ over $\K(\bar x,\bar y)$ ramifies precisely at the places of $\K(\bar x,\bar y)$ lying over $\gamma_2,\gamma_4,\gamma_5,\gamma_6$. Let $\Omega_j$ be the set of places of $\K(\bar x,\bar y,\bar w)$ lying over $\gamma_j$. Note that $|\Omega_1|=|\Omega_3|=2m$ and 
$|\Omega_j|=m$ for each $j$ in $\{2,4,5,6\}$. From Lemma \ref{L14} we have that in $\K(\bar x,\bar y,\bar w)$
\begin{eqnarray*}
\div (\bar x)=\sum_{\gamma \in \Omega_4\cup \Omega_5}2\gamma-\sum_{\gamma \in \Omega_ 2}2\gamma-\sum_{\gamma \in \Omega_1}\gamma,\\
\div (\bar y)=\sum_{\gamma \in \Omega_2\cup \Omega_6}2\gamma-\sum_{\gamma \in \Omega_ 4}2\gamma-\sum_{\gamma \in \Omega_3}\gamma.
\end{eqnarray*}
Also, by \eqref{23genn},
$$
\div (\bar w)=m\Big(\sum_{\gamma\in \Omega_5\cup \Omega_6}\gamma-\sum_{\gamma\in \Omega_2\cup \Omega_4}\gamma\Big).
$$ 
 As regards  $\bar x^m-\bar y^m=\bar u-\bar z$, it is easily seen that in $\K(\bar u, \bar z)$ the rational function $\bar u-\bar z$ has at most $4$ distinct zeros; hence, the set $E'$ of zeros of $\bar x^m-\bar y^m$ in $\K(\bar x,\bar y,\bar w)$ has size at most $8m^2$. Clearly any place of $E$ is contained either in $E'$ or in $\Omega_j$ for some $j=1\ldots,6$, whence $|E|\le 8m^2+8m$.
 
Our assumption on $q$ and $m$, together with Proposition \ref{HaWe}, ensures the existence of at least $8m^2+8m+1$ $\fq$-rational places of $\K(\bar x,\bar y,\bar w)$; hence,
there exists at least one $\fq$-rational place $\gamma_c$ of $\K(\bar x,\bar y,\bar w)$ not in $E$.
Let 
$$
\tilde x=\bar x(\gamma_c),\quad \tilde y=\bar y(\gamma_c), \quad \tilde w=\bar w(\gamma_c).
$$
Note that $P_c=(\tilde x,\tilde y)$ is an $\fq$-rational affine point of the curve with equation $f_{a,b,t,m}(X,Y)=0$. Therefore, by Proposition \ref{allinea}, 
$P$ is collinear with two distinct points
$$
P_{1,c}=\Big(t\tilde x^m, \frac{(t\tilde x^m-1)^3}{t\tilde x^m}\Big),\,P_{2,c}=\Big(t\tilde y^m, \frac{(t\tilde y^m-1)^3}{t\tilde y^m}\Big)\in K_t.
$$
If $c$ is chosen to be a square, then $P$ is external to $P_{1,c}P_{2,c}$; on the other hand, if $c$ is not a square, then $P$ is internal to $P_{1,c}P_{2,c}$. 
This proves the assertion.
  \end{proof}
 
As $m>2$ the coset $K_t$ cannot bicover all the $\fq$-rational affine points in $\cX$. Therefore, unions of  distinct cosets need to be considered.
%Assume that $K_{t'}$ is a coset of $K$ such that $K_t\cup K_{t'}$ is an arc, and let  %$P_0=(u_0,(u_0-1)^3/u_0)$ with $u_0 \neq 0$ be an $\fq$-rational affine point of %$\cX$ not belonging to $K_t\cup K_{t'}$ but collinear with a point of $K_t$ and a %point of $K_{t'}$.
\begin{proposition}\label{10feb} 
Let $K_{t'}$ be a coset of $K$ such that $K_t\cup K_{t'}$ is an arc. Let $P_0$  be an $\fq$-rational affine point of $\cX$ not belonging to $K_t\cup K_{t'}$ but collinear with a point of $K_t$ and a point of $K_{t'}$.
If \eqref{aritmetica1} holds, then $P_0$ is bicovered by $K_t\cup K_{t'}$. 
\end{proposition}
\begin{proof}
Let $P_0=(u_0,(u_0-1)^3/u_0)$ with $u_0 \neq 0$.
Note that when $P$ ranges over  $K_t$, then the point $Q=\ominus (P_0\oplus P)$ is collinear with  $P_0$ and ranges over $K_{t'}$.  
Recall that  $P$ belongs to $K_t$ if and only if  
$$
P=\Big(tx^m,\frac{(tx^m-1)^3}{tx^m}\Big)
$$ 
 for some $x\in \fq^*$. In this case,  
$$
Q=\Big(\frac{1}{u_0tx^m},\frac{(1-u_0tx^m)^3}{(u_0tx^m)^2}\Big).
$$
Let $\bar x$ be a transcendental element over $\K$.
% let 
%$$
%e(\bar x)=\beta\frac{\bar t(\frac{\bar x+\beta}{\bar x-\beta})^m+1}{\bar t (\frac{\bar x+\beta}{\bar %x-\beta})^m-1}=\frac{\beta\bar t (\bar x+\beta)^m+\beta(\bar x-\beta)^m}{\bar t (\bar x+\beta)^m-(\bar %x-\beta)^m}\in \K(\bar x).
%$$
%Note that $e(\bar x)$ is defined over $\fq$.
In order to determine whether $P_0$ is bicovered by $K_t\cup K_{t'}$ we need to investigate whether the following rational function is a non-square in $\K(\bar x)$:
$$
\eta(\bar x)=(u_0-t\bar x^m)\Big(u_0-\frac{1}{u_0tx^m}\Big)=\frac{(u_0-t\bar x^m)(u_0^2t\bar x^m-1)}{u_0t\bar x^m}.
$$
Let $\gamma_0$ and $\gamma_\infty$ be the zero and the pole of $\bar x$ in $\K(\bar x)$, respectively.
Note that both $\gamma_0$ and $\gamma_\infty$ are poles of $\eta(\bar x)$ of multiplicity $m$, since $\gamma_\infty$ is a pole of order $m$ of $(u_0- t{\bar x}^m)$, $(u_0^2t{\bar x}^m-1)$, and $ u_0t{\bar x}^m$; hence, $v_{\gamma_{\infty}}(\eta(\bar x))=-m-m-(-m)=-m$. Also,  $\gamma_0$ is a zero of $u_0t {\bar x}^m$ of multiplicity $m$. As $m$ is odd,  $\eta(\bar x)$ is not a square in $\K(\bar x)$.
Then Proposition \ref{teo1cor} applies to $c\eta(\bar x)$ for each $c\in \fq^*$. Since $\eta(\bar x)$ has exactly two poles, and the number  of its zeros is at most $2m$, the genus of the Kummer extension $\K(\bar x,\bar w)$ of $\K(\bar x)$ with $\bar w^2=c\eta(\bar x)$ is at most $m$.

Our assumption on $q$, together with the Hasse-Weil bound, yield the existence of an $\fq$-rational place $\gamma_c$ of $\K(\bar x,\bar w)$ which is not a zero nor a pole of $\bar w$. 
Let 
$
\tilde x=\bar x(\gamma_c)$, $ \tilde w=\bar w(\gamma_c).
$
Therefore,
$P_0$ is collinear with two distinct points
$$
P(c)=\left(t\tilde x^{m},\frac{(\tilde x^{m}-1)^3)}{t \tilde x ^{m}}\right) \in K_{t},\qquad
Q(c)=\left(\frac{1}{u_0t\tilde x^m},\frac{(1-u_0t\tilde x^m)^3}{(u_0t\tilde x^m)^2}\right)
\in K_{t'}.
$$
If $c$ is chosen to be a square, then $P_0$ is external to $P(c)Q(c)$; on the other hand, if $c$ is not a square, then $P_0$ is internal to $P(c)Q(c)$. 
 \end{proof}

In order to construct bicovering arcs contained in $\cX$, the notion of a maximal-$3$-independent subset of a finite abelian group $\mathcal G$ is needed, as given in \cite{MR1075538}.
A subset $M$ of $\mathcal G$ is said to be {\em maximal }$3$-{\em independent} if 
\begin{itemize}
\item[ (a)] $x_1+x_2+x_3\neq 0$ for all $x_1,x_2,x_3\in M$, and 
\item[(b)] for each $y\in \mathcal G\setminus M$ there exist $x_1,x_2\in M$ with $x_1+x_2+y=0$. 
\end{itemize}
If in (b) $x_1\neq x_2$ can be assumed, then $M$ is said to be {\em good}.
Now, let $M$ be a maximal $3$-independent subset of the factor group $G/K$ containing $K_t$. Then
the union $S$ of the cosets of $K$ corresponding to  $M$  is a good maximal $3$-independent subset of $(G,\oplus)$; see \cite{MR1075538}, Lemma 1, together with Remark 5(5). In geometrical terms, since three points in $G$ are collinear if and only if their sum is equal to the neutral element,
 $S$ is an arc whose secants cover all the points in $G$. By Propositions \ref{contiamo} and \ref{10feb}, if $K$  is large enough with respect to $q$ then $S$ is a bicovering arc as well, and the following result holds. 
\begin{theorem}\label{pre14mar} Let $m$ be a proper divisor of $q-1$ such that $(m,6)=1$ and \eqref{aritmetica1}
holds. Let $K$ be a subgroup of $G$ of index $m$. For $M$  a maximal $3$-independent subset of the factor group $G/K$, the point set
$$
S=\bigcup_{K_{t_i}\in M}K_{t_i}
$$ 
is a bicovering arc in $AG(2,q)$ of size $\#M\cdot \frac{q-1}{m}$.
\end{theorem}

\section{Conclusions}\label{conc}
%Theorem \ref{pre14mar} has the following straightforward corollary.

%\begin{corollary}\label{14mar} Let $m$ be a proper divisor of $q-1$ such that $(m,6)=1$ and %\eqref{aritmetica1}
%holds. Assume that the cyclic group of order $m$ admits a maximal $3$-independent subset of %size $s$. Then there exists a bicovering arc in $AG(2,q)$ of size $\frac{s(q-1)}{m}$.
%\end{corollary}

We use Theorem \ref{pre14mar}, together with the results in Section \ref{bicovar} in order to construct small complete caps in affine spaces $AG(N,q)$.
Note that \eqref{aritmetica1}
holds when
$$
\sqrt q\ge 6m^2-4m+1+\sqrt{36m^4-48m^3+36m^2+1},
$$
which is clearly implied by
$
m\le \frac{\sqrt[4]q}{3.5}.
$

\begin{corollary}\label{lastbut1} Let $m$ be a proper divisor of $q-1$ such that $(m,6)=1$ and 
$m\le \frac{\sqrt[4]q}{3.5}$. Assume that the cyclic group of order $m$ admits a maximal $3$-independent subset of size $s$. Then 
\begin{itemize}
\item[\rm{(i)}] there exists a bicovering arc in $AG(2,q)$ of size $\frac{s(q-1)}{m}$;
\item[\rm{(ii)}] for $N\equiv 0 \pmod 4$, $N\ge 4$, there exists a complete cap in $AG(N,q)$ of size 
$$
\frac{s(q-1)}{m}q^{\frac{N-2}{2}}.
$$
\end{itemize}
\end{corollary}

In the case where a group $\mathcal G$ is the direct product of two groups $\mathcal G_1$, $ \mathcal G_2$ of order at least $4$, neither of which elementary $3$-abelian, there exists a maximal $3$-independent subset of $\mathcal G$ of size less than or equal to
$(\#\mathcal G_1)+(\#\mathcal G_2)$; see \cite{MR1221589}. Then Theorem  \ref{duedue} follows at once from Corollary \ref{lastbut1}.

%the following holds.
%
%\begin{theorem}\label{corfinale} Let $m$ be a proper divisor of $q-1$ such that $(m,6)=1$ %and 
%$m\le \frac{\sqrt[4]q}{3.5}$. Assume that  $m=m_1m_2$ with $(m_1,m_2)=1$ and $m_1,m_2>3$. %Then for $N\equiv 0 \pmod 4$, $N\ge 4$, there exists a complete cap in $AG(N,q)$ of size %less than or equal to
%$$
%\frac{(m_1+m_2)(q-1)}{m_1m_2}q^{\frac{N-2}{2}}.
%$$
%\end{theorem}

\subsection{Comparison with previous results}\label{confronto}

We distinguish two possibilities for the integer $h$ such that $q=p^h$.

%Note that when $p$ is large, the value $(m_1+m_2)/m_1m_2$,
%where $m_1$, $m_2$ are coprime divisors of $q-1$
%with $(m_1m_2,6)=1$ and $m_1m_2<\sqrt[4]{q}/3.5$,
%can be significantly smaller than both $p/q^{1/8}$ and $1/3$.

\subsubsection{$h\le 8$} 

The best previously known general construction of complete caps in $AG(N,q)$ is that given in \cite{Anb-Giu}, providing complete caps of size approximately $q^{N/2}/3$.
It is often possible to choose $m_1$, $m_2$  as in Theorem \ref{duedue} 
%coprime divisors of $q-1$
%with $(m_1m_2,6)=1$ and $m_1m_2<\sqrt[4]{q}/3.5$,
in such a way that the value $(m_1+m_2)/m_1m_2$ is significantly smaller than $1/3$. 

This happens for instance for all $q=p^h$ such that $p-1$ has a composite divisor $m<\sqrt[4]{p}/3.5$ with $(m,6)=1$.

For $p>3$ generic, when $h=8$
a possible choice for $m$ is $m=(p^2-1)/(2^{s}3^{k})$, where $2^{s}\ge 4$ is the highest power of $2$ which divides $p^2-1$, and similarly $3^{k}\ge 3$ is the highest power of $3$ which divides $p^2-1$. 
%In fact, both $m\le \frac{\sqrt[4]q}{3.5}$ and $(m,6)=1$ hold.
Assume first that $3$ divides $p-1$, so that $(3,p+1)=1$. 
Then $m=m_1m_2$ where $m_1=(p-1)/(2^{s_1}3^k)$ and $m_2=(p+1)/2^{s_2}$ with $s_1+s_2=s$. %Since $\gcd(m_1,m_2)=1$, there exists a maximal $3$-independent subset of size at most %$m_1+m_2$ of the cyclic group of order $m$. For instance 
%When $s_1=s_2=k=1$, Corollary \ref{lastbut1} provides bicovering arcs of size less than or $equal to 
%$$
%$\frac{4p+2}{6}\cdot (q-1)\cdot \frac{12}{p^2-1}\sim 8q^{\frac{7}{8}},
%$$
%and hence complete caps in $AG(N,q)$ of size less than or equal to
%$
%8q^{\frac{N}{2}-\frac{1}{8}}.
%$
%If either $s>2$ or $k>1$ then we obtain 
Then Theorem \ref{duedue} provides complete caps in $AG(N,q)$ of size 
 approximately at most
$$
(2^{s_2}+2^{s_1}3^k)q^{\frac{N}{2}-\frac{1}{8}}.
$$
If $3$ divides $p+1$ a similar bound can be obtained.

\subsubsection{$h>8$} 

The smallest known complete caps in $AG(N,q)$ have size approximately
 $$2q^{N/2}/p^{\lfloor (\lceil h/4 \rceil-1)/2 \rfloor};$$ see \cite[Theorem 6.2]{ABGP}.
Theorem \ref{duedue} provides an improvement on such bound whenever it is possible to choose $m_1,m_2$ so that 
\begin{equation}\label{11aprile}(m_1+m_2)/m_1m_2 < 2/p^{\lfloor (\lceil h/4 \rceil-1)/2 \rfloor}.\end{equation}

This certainly happens for instance when $h\equiv 0 \pmod 8$ and $p$ is large enough. Let $2^{s}\ge 4$ be the highest power of $2$ which divides $\sqrt[4]{q}-1$, and similarly $3^{k}\ge 3$  the highest power of $3$ which divides $\sqrt[4]{q}-1$.
Then arguing as in Case (i) it is easy to see that one can choose $m_1,m_2$ so that
$$
\frac{m_1+m_2}{m_1m_2} \sim (2^{s_2}+2^{s_1}3^k)q^{-\frac{1}{8}},
$$
with $s_1+s_2=s$. On the other hand, $2/p^{\lfloor (\lceil h/4 \rceil-1)/2 \rfloor}=2pq^{-\frac{1}{8}}$.

Another family of $q$'s for which \eqref{11aprile} happens is $q=p^{12}$, with $p\equiv 1 \pmod {12}$ and $(p^2+1)/2$ a composite integer. Assume that $(p^2+1)/2=v_1v_2$ with $v_1,v_2>1$ and $v_1<v_2$. Then choosing $m_1=v_1(p+1)/2$ and $m_2=v_2$ gives
$
(m_1+m_2)/m_1m_2<2/p. 
$   
\begin{comment}

Here  $q$ is an $8$-th power, say $q=(q')^8$, and a possible choice for $m$ is $m=((q')^2-1)/(2^{h}3^{k})$, where $2^{h}\ge 4$ is the highest power of $2$ which divides $(q')^2-1$, and similarly $3^{k}\ge 3$ is the highest power of $3$ which divides $(q')^2-1$. 
In fact, both $m\le \frac{\sqrt[4]q}{3.5}$ and $(m,6)=1$ hold.

Assume that $3$ divides $q'-1$, so that $(3,q'+1)=1$. 
Then $m=m_1m_2$ where $m_1=(q'-1)/(2^{h_1}3^k)$ and $m_2=(q'+1)/2^{h_2}$ with $h_1+h_2=h$. %Since $\gcd(m_1,m_2)=1$, there exists a maximal $3$-independent subset of size at most %$m_1+m_2$ of the cyclic group of order $m$. For instance 
When $h_1=h_2=k=1$, Theorem \ref{duedue} provides bicovering arcs of size less than or equal to 
$$
\frac{4q'+2}{6}\cdot (q-1)\cdot \frac{12}{(q')^2-1}\sim 8q^{\frac{7}{8}},
$$
and hence complete caps in $AG(N,q)$ of size less than or equal to
$
8q^{\frac{N}{2}-\frac{1}{8}}.
$
If either $h>2$ or $k>1$ then we obtain complete caps of size approximately at most
$$
(2^{h_2}+2^{h_1}3^k)q^{\frac{N}{2}-\frac{1}{8}}.
$$
If $3$ divides $q'+1$ similar bounds are obtained.

We remark that for $h=8$ then XXXXX, whereas for $h>8$ XXXXX

\item[(iii)]
\end{comment}

\end{document}